\documentclass{amsart}
\usepackage{amsmath, amsthm}
\newtheorem{theorem}{Theorem}[section]
\newtheorem{lemma}[theorem]{Lemma}

\theoremstyle{definition}

\newtheorem{example}[theorem]{Example}
\newtheorem{proposition}[theorem]{Proposition}
\newtheorem{remark}[theorem]{Remark}
\newtheorem{corollary}[theorem]{Corollary}
\numberwithin{equation}{section}

\begin{document}

\title[Fixed Loci of The Anticanonical Complete Linear Systems ...]{ Fixed Loci  of the Anticanonical Complete Linear Systems of Anticanonical Rational Surfaces}
\author[\sc Cerda Rodr\'{\i}guez]{Jes\'us Adrian Cerda Rodr\'{\i}guez} \address{J. A. Cerda Rodr\'{\i}guez.   Instituto de
F\'{\i}sica y Matem\'aticas (I.F.M.) Universidad Michoacana de San
Nicol\'as de Hidalgo (U.M.S.N.H.). Edificio C-3, Ciudad
Universitaria. C. P.
58040 Morelia, Michoac\'an, M\'exico.\\
E-mail: jacerda@ifm.umich.mx}

\author[\sc Failla]{Gioia Failla}
\address{G. Failla. Department of Mathematics of the University of Messina, C.da
Sant'Agata, Salita Sperone.  98166-Messina, Italy.\\
E-mail: gfailla@dipmat.unime.it}

\author[\sc Lahyane]{Mustapha Lahyane}\address{M. Lahyane. Instituto de
F\'{\i}sica y Matem\'aticas (I.F.M.) Universidad Michoacana de San
Nicol\'as de Hidalgo (U.M.S.N.H.). Edificio C-3, Ciudad
Universitaria. C. P. 58040 Morelia,
Michoac\'an, M\'exico.\\
E-mail: lahyane@ifm.umich.mx}

\author[\sc Osuna Castro]{Osvaldo Osuna Castro}\address{O. Osuna Castro. Instituto de F\'{\i}sica y Matem\'aticas (I.F.M.) Universidad Michoacana de San
Nicol\'as de Hidalgo (U.M.S.N.H.). Edificio C-3, Ciudad
Universitaria. C. P.
58040 Morelia, Michoac\'an, M\'exico.\\
E-mail:  osvaldo@ifm.umich.mx}

\thanks{Research partially supported in 2011 by C.I.C. - U.M.S.N.H. (Mexico).}

\date{24 November 2011}



\begin{abstract}
We determine the fixed locus of the anticanonical complete linear
system of a given anticanonical rational surface. The case of a
geometrically ruled rational surface is fully studied, e.g., the
monoid of numerically effective divisor classes of such surface is
explicitly determined and is minimally generated by two elements. On
the other hand, as a consequence in the particular case where $X$ is
a smooth rational surface with $K_{X}^{2}>0$, the following expected
result holds: every fixed prime divisor of the complete linear
system $|-K_X|$ is a $(-n)$-curve, for some integer $n\geq 1$.
\end{abstract}

\maketitle

\section{Introduction}

This note is mainly devoted to determine the integral curves of
the fixed locus of the complete linear system $|-K_X|$ of an
anticanonical rational surface $X$. Here $X$ is anticanonical
means that it is  smooth and such that the complete linear system
$|-K_X|$ is not empty, where $K_{X}$ denotes a canonical divisor
on $X$. Such linear system is worth studying, for example,
Hironaka considers the unique fixed irreducible component of the
anticanonical complete linear system of a very special
anticanonical  rational surface in order to give an example for
which the contraction of an integral curve of strictly negative
self-intersection on an algebraic surface is not necessarily an
algebraic one (this contraction is always an analytic surface
according to Grauert).

From Theorem $\ref{FCAD2}$ below, it appears that if the fixed
locus is not the zero divisor - such situation is the general one
- then its irreducible components are either smooth rational
curves of strictly negative self-intersection or an integral curve
of arithmetic genus equal to one which has  in almost all cases a
strictly negative self-intersection. The case where the fixed
locus is zero implies that $-K_X$ is numerically effective. The
nef-ness  condition of  $-K_X$ means that the intersection number
of $K_X$ and of any prime divisor on $X$ is less than or equal to
zero.  Thus the inequality $K_{X}^{2}\geq 0$ holds and
consequently the Picard number $\rho{(X)}$ of $X$ is less than or
equal to ten.

On the other hand, from the Riemann-Roch Theorem (see Lemma
\ref{anti} below), a smooth rational surface $Y$ having a canonical
divisor $K_{Y}$ of self-intersection greater than or equal to zero
is anticanonical. Such surfaces are studied intensively for
different reasons in \cite{MN}, \cite{JR}, \cite{UPRM},
\cite{Harb2}, \cite{L1},  \cite{L2}, \cite{L3} and \cite{LH}. The
case where the self-intersection of a canonical divisor is equal to
zero is very special and leads to very interesting geometric
phenomena, see for instance \cite{MN}, \cite{UPRM}, \cite{L2},
\cite{L3} and \cite{LH}. Finally, when the self-intersetion of 
a canonical divisor is negative, one may determine the geometry 
of some specific projective rational  surfaces, e.g. see \cite{L0}, 
\cite{L2}, \cite{L3},  \cite{FLM0}, \cite{FLM1}, \cite{FLM2}, 
\cite{GM1} and \cite{GM2}.

In the case where $K_{Y}^{2}>0$ and if the fixed locus of the
complete linear system $|-K_{Y}|$ is not equal to zero as a divisor,
we will deduce mainly from Theorem $\ref{FCAD2}$ that its prime
components are smooth rational curves of strictly negative
self-intersection (see Corollary \ref{FCAD3} below). Whereas in the
case where $K_{Y}^{2}=0$, it may happen that $|-K_{Y}|$ is equal to
a singleton, so in particular, the fixed locus is an integral curve
of arithmetic genus equal to one and of self-intersection equal to
zero.

This note is organized as follows. In section $2$, we give some
standard facts about smooth rational surfaces and fix our
notations. Section $3$ deals with the case when the Picard number
of the smooth rational surface is equal to two, i.e., the case of
geometrically ruled rational surfaces. We determine the fixed
locus of the complete linear system associated to any effective
divisor (see Proposition \ref{FCGRSED}). Also, in this case, the
monoid of numerically effective divisor classes of the
geometrically ruled rational surface is explicitly determined, it
is shown that it is minimally generated by two elements, see Lemma
\ref{Monoid}. Finally, section $4$ contains our main result (see
Theorem \ref{FCAD2} below). It is shown that if the fixed locus of
the anticanonical rational surface is not equal  to zero, then
every integral curve of the fixed locus is either a $(-n)$-curve
for some integer $n\geq 1$, or an integral curve of arithmetic
genus equals to one and of self-intersection less than or equal to
zero. Whereas if the fixed locus is zero, then the
self-intersection of the canonical divisor of the surface is
larger than or equal to zero; thus gives an explicit description
of the anticanonical rational surface.

\section{Preliminaries}

In this section, we mention the notions that we need. See \cite{HA}
as a  reference for these materials. Let $X$ be a smooth algebraic
surface defined over an algebraically closed field. A divisor on $X$
is effective if it is a nonnegative linear combination of prime
divisors. Similarly, a class of divisors modulo algebraic
equivalence on $X$ is effective if it contains an effective divisor.
Moreover, if $X$ is rational, then the class of divisors modulo
algebraic equivalence containing the divisor $D$ on $X$ is effective
if and only if the vector space of global sections of the invertible
sheaf ${\mathcal O_{X}(D)}$ associated to $D$ in the Picard group
$Pic(X)$ of $X$ is not trivial.  Indeed, more generally  the
algebraic, the linear, the numerical and the rational equivalences
of divisors on the smooth rational surface $X$ are the same.  On the
other hand $Pic(X)$ is isomorphic to the group $Cl(X)$ of classes of
divisors modulo linear equivalence on $X$.\\

Let $Y$  be  an anticanonical rational surface and let $K_{Y}$ be
a canonical divisor on it. That $Y$ is anticnonical means  by
definition that $Y$ a smooth surface such that its anticanonical
complete linear system $|-K_{Y}|$ is not empty. Following
\cite{HA}, we adopt in all this note the following notations:
\begin{itemize}

\item $Div(Y)$ is the group of divisors on $Y$.

\item $D\sim D'$ means that $D$ is linearly equivalent to $D'$,
where $D$ and $D'$ are elements of $Div(Y)$,

\item $Cl(Y)$ is the quotient group $Div(Y)/\sim $ of $Div(Y)$ by $\sim $.

\item $NS(Y)$ is the N\'eron-Severi group $NS(Y)$ of $Y$, i.e., the quotient
group of $Div(Y)$ by the numerical equivalence classes of divisors
on $Y$. Since $Y$ is a rational surface, the linear and numerical
equivalences are equivalents on $Div(Y)$. One has  $NS(Y)$ is
equal to $Cl(Y)$.

\item $\rho (Y)$ is the rank of $NS(Y)$ and called the Picard
number of $Y$.

\item $\Bbb{F}_{n}$ is the Hirzebruch
surface associated to the integer $n$, $n\geq 0$ (see
\cite[Section 2, p. 369 ]{HA}).

\item $\mathcal{F}$ is the element of $NS(\Bbb{F}_{n})$ associated
to any fiber of the ruling of $\Bbb{F}_{n}$ if $n\neq 0$, and any
fiber of any ruling of $\Bbb{F}_{0}$ if $n=0$.

\item $\mathcal{C}{_n}$ is the element of $NS(\Bbb{F}_{n})$ determined by
the unique integral curve of self-intersection equal to $-n$ if $n
\neq 0$ or any fiber $F'$ of the second ruling if $n=0$.

\item For a smooth rational surface $Y$, $\rho (Y)=1$ if and only if
$Y$ is isomorphic to the projective plane $\Bbb P ^{2}$. And $\rho
(Y)=2$ if and only if $Y$ is isomorphic to $\Bbb{F}_{n}$ for some
$n\geq 0$. This can be deduced from \cite[Chapter 5 ]{HA}).
\end{itemize}

Now we state the Riemann-Roch Theorem for smooth algebraic
surfaces, see  \cite[Theorem 1.6 (Riemann-Roch)., page 362]{HA}.

Let $X$ be a smooth algebraic surface. If $D$ is a divisor on $X$
and $\mathcal O_{X}(D)$ denotes the invertible sheaf associated to
$D$ in $Pic(X)$. Then the following equality holds.
$$h^{0}(X,{\mathcal O_{X}(D)})- h^{1}(X,{\mathcal O_{X}(D)})+
h^{0}(X,{\mathcal O_{X}(K_{X}-D})) = \chi ({\mathcal O_{X}})+
\frac{1}{2}(D^{2}-K_{X}.D),$$ where $K_{X}$ and $\chi (X)$ denote
a canonical divisor and the Euler characteristic of $X$
respectively.

Notice that for smooth rational surfaces $Z$, one always has $\chi
({\mathcal O_{Z}})=1$.

The next lemma provides, in particular,  an example of an
anticanonical rational surface. It is a straightforward application
of the Riemann-Roch Theorem  to the invertible sheaf associated to
an anticanonical divisor (see \cite[Theorem1.6 (Riemann-Roch)., page
362]{HA})  and of the rationality criterion of Castelnuevo (see
\cite[Theorem 6.1., page 422]{HA} and \cite{BPV}).

\begin{lemma}\label{anti}
Let $Y$ be a smooth rational surface such that $K_{Y}^{2}\geq 0$.
Then $Y$ is anticanonical.
\end{lemma}

Here we recall the notion of  nefness of divisors on a smooth
algebraic surface $X$.

Let $D$ be a divisor on a smooth algebraic surface $X$. $D$ is
numerically effective (nef in short) if the intersection number of
$D$ with any prime divisor on $X$ is larger than or equal to zero.
Similarly, a class of divisors modulo algebraic equivalence on $X$
is nef if this class contains a nef divisor.

To illustrate the last definition, the following examples are
useful:
\begin{example}\label{nef1}
Let $ \pi : X \longrightarrow \Bbb P^{2}$ be the blow up the
projective plane $\Bbb P^{2}$ at a finite set of points. Then, the
class of a line pulled back to $X$ via $\pi $ is  nef. However,
the exceptional divisors  are not nef.
\end{example}

\begin{example}\label{nef2}
Let $\Bbb{F}_{n}$ be the Hirzebruch surface associated to the
integer $n\geq 0$. Then   $\mathcal{F}$ and $(\mathcal{C}{_n} + n
\mathcal{F})$ are numerically effective.
\end{example}

The following example generalizes the useful remark stated in
\cite[Remarque utile III.5, p.35]{B}.
\begin{example}\label{nef3}
Let $Z$ be a smooth algebraic surface. Let $\Gamma _{1}, \ldots ,
\Gamma _{p}$ be the irreducible components of the effective divisor
$D$ on $Z$. Then the followings are equivalents:
\begin{enumerate}
\item The intersection number of $D$ and $\Gamma _{i}$ is larger
than or equal to zero for every $i=1, \ldots , p$.

\item  $D$ is nef.

\end{enumerate}
\end{example}

 We are interested to answer the following question: let $Y$ be
an anticanonical rational surface and let $K_{Y}$ be a canonical
divisor on $Y$. What kind of fixed integral curves may have the
anticanonical complete linear system $|-K_{Y}|$ if it has some?
More specially, we are interested in the curves which are fixed
components in $|-K_{Y}|$.\\

Since the anticanonical complete linear system
$|{\mathcal{O}}_{\Bbb{P}^{2}}(3)|$ of the projective plane
$\Bbb{P}^{2}$ does not have a fixed component, we will focus in
the case $\rho (Y)\geq 2$.  Firstly in the next section, we will
review the case of geometrically ruled rational surfaces, i.e.,
those smooth rational surfaces with Picard number equal to two.\\

Here we give a useful result.

\begin{lemma}
\label{move} Let $\Gamma $ be a prime divisor on an anticanonical
rational surface $Z$. If $\Gamma ^{2}>0$, then $h^{0}(Z,
O_{Z}(\Gamma ))\geq 2$.
\end{lemma}

\begin{proof}
The Riemann-Roch Theorem applied to the invertible sheaf
${\mathcal O}_{Z}(\Gamma )$ gives the following inequality:
$$h^{0}(Z,{\mathcal O}_{Z}(\Gamma ))\geq 1+ \frac{1}{2}(\Gamma
^{2}-K_{Z}.\Gamma ).$$ An application of Example \ref{nef3} to
$\Gamma $ shows that $\Gamma $ is nef. Taking into account that
$Z$ is anticanonical leads to the inequality: $\Gamma . K_{Z}\leq
0$. Then the  result follows obviously if $\Gamma . K_{Z}\leq -1$.
Whereas if $\Gamma . K_{Z}=0$, then the adjunction formula implies
that $\Gamma ^{2}\geq 2$. And we are done.
\end{proof}

\begin{remark}  If one allows that
$\Gamma ^{2}=0$  in the above Lemma \ref{move}, then the
inequality $h^{0}(Z, O_{Z}(\Gamma ))\geq 2$ may fail to hold.
\end{remark}

\section{The Case of a Geometrically Ruled Rational Surface}

Let $\Bbb{F}_{n}$ be the Hirzebruch surface associated to the
integer $n\in \Bbb N$. The N\'eron-Severi group $NS(\Bbb{F}_{n})$
of $\Bbb{F}_{n}$ is a free abelian group generated by
$\mathcal{C}{_n}$ and $\mathcal{F}$ and it is endowed with the
intersection form denoted by . which is given  on the generators
by (see \cite[proposition 3.2., p. 386]{HA}):

\begin{itemize}

\item [$\bullet \;\;\;$] $\mathcal{C}{_n}^{2}=-n$;
\item [$\bullet \;\;\;$] ${\mathcal F}^{2}=0$;
\item [$\bullet \;\;\;$]  $\mathcal{C}{_n}.\mathcal F=1$.
\end{itemize}

The following lemma shows that $\mathcal{C}{_n}$ and $\mathcal{F}$
generate also the monoid $M(\Bbb{F}_{n})$ of effective divisor
classes of $\Bbb{F}_{n}$ and that the monoid $NEF(\Bbb{F}_{n})$ of
numerically effective divisors classes of $\Bbb{F}_{n}$ is
generated by  two elements, namely $(\mathcal{C}{_n} + n
\mathcal{F})$ and $\mathcal{F}$. Note that both $\mathcal{C}{_n}$,
$(\mathcal{C}{_n} + n \mathcal{F})$  and $\mathcal{F}$ are all of
them prime  classes, i.e., each of them is the class in
$NS(\Bbb{F}_{n})$ of a prime divisor on $\Bbb{F}_{n}$. For
completeness, we give a proof of it.

\begin{lemma}
\label{Monoid} Let  $NS(\Bbb{F}_{n})$ be as above. Then
\begin{itemize}
\item[1.] $M(\Bbb{F}_{n})= \Bbb N \mathcal{C}{_n} + \Bbb N
\mathcal{F}$. Moreover, $M(\Bbb{F}_{n})$ can not be generated by
one element.
\item[2.] $NEF(\Bbb{F}_{n})= \Bbb N (\mathcal{C}{_n} + n \mathcal{F}) +
\Bbb N \mathcal{F}$. Moreover, $NEF(\Bbb{F}_{n})$ can not be
generated by one element.
\end{itemize}
\end{lemma}

\begin{proof}

\begin{itemize}
\item[1.] The inclusion $\Bbb N \mathcal{C}{_n} + \Bbb N
\mathcal{F}\subset M(\Bbb{F}_{n})$ is clear. Let us see why the
other inclusion is true. Take an element $z$ in $M(\Bbb{F}_{n})
\subset NS(\Bbb{F}_{n})$, it follows that $z= u\mathcal{C}{_n} +
v\mathcal{F}$ for some integers $u$ and $v$. The fact that
$\mathcal{F}$ and $(\mathcal{C}{_n}+n\mathcal{F})$ (see Example
$\ref{nef2}$) are numerically effective gives the required
inequalities $u=z. \mathcal{F}\geq 0$ and $v=
z.(\mathcal{C}{_n}+n\mathcal{F})\geq 0$. This proves the first
statement. Since $\mathcal{C}{_n}$ and $\mathcal{F}$ are linearly
independents, the submonoid $M(\Bbb{F}_{n})$ of $NS(\Bbb{F}_{n})$
can not be generated by one element.

\item[2.]
It is obvious that $\Bbb N (\mathcal{C}{_n} + n \mathcal{F}) + \Bbb
N \mathcal{F} \subset NEF(\Bbb{F}_{n}) $. Now, let $x$ be an element
of $NEF(\Bbb{F}_{n}) \subset  NS(\Bbb{F}_{n})$, there exist then two
integers $a$ and $b$ such that $x= a\mathcal{C}{_n} + b\mathcal{F}$.
Since $\mathcal{F}$ and $\mathcal{C}{_n}$ are effective and $x$ is
numerically effective, we get $0\leq \mathcal{F}. x=a$ and $0\leq
x.\mathcal{C}{_n}=b-na$. So, $x=a\mathcal{C}{_n} + b\mathcal{F}=
a(\mathcal{C}{_n}+n\mathcal{F})+ (b-na)\mathcal{F}$ and we are done.
Again as above,$NEF(\Bbb{F}_{n})$ can not be generated by one
element.
\end{itemize}
\end{proof}

Next, we determine the fixed locus of any complete linear system
$|aC_n+bF|$ associated to an effective divisor $D_{(a,b)}$ whose
class in the N\'eron-Severi group $NS(\Bbb{F}_{n})$  is
$a\mathcal{C}_{n} + b\mathcal{F}$. Our result is:

\begin{proposition}
\label{FCGRSED} Let $a\mathcal{C}_{n} + b\mathcal{F}$ be an
effective element of $NS(\Bbb{F}_{n})$, where $n$ is an integer
greater than or equal to zero. Then, the complete linear system
$|aC_n+bF|$ does not have a fixed component if both inequalities
$b\geq an$ and $n\geq 1$ hold. Moreover if $b<an$, then  there is
only one fixed component. In this case the fixed component and the
mobile component of $|aC_n+bF|$ are $jC_{n}$ and $(a-j)C_{n}+bF$
respectively, where $j$ is  the unique integer $j$ such that
$1\leq j\leq a$ and $(a-j)n\leq b\leq (a-j+1)n-1$. For $n=0$,
$a\mathcal{C}_{0} + b\mathcal{F}$ does not have a fixed component.
\end{proposition}

\begin{proof} Assuming that $b\geq an$ and  $n\geq 1$, it follows
from \cite[Corollary 2.18., page 380]{HA}  that the complete
linear system $|aC_n+bF|$ does not have a fixed component. Now if
$b<an$, then from $jC_n.(aC_n+bF)=j(b-an)<0$ we deduce that $jC_n$
is a fixed component of $|aC_n+bF|$, even it is the fixed
component since $|(a-j)C_n+bF|$ contains an integral curve. To end
the proof, it is straightforward from \cite[Corollary 2.18., page
380]{HA} that if $n=0$, then the effective class $a\mathcal{C}_{0}
+ b\mathcal{F}$ has a zero fixed locus.
\end{proof}

A direct application of the last proposition to an anticanonical
divisor $2C_n+(2+n)F$ on $\Bbb{F}_{n}$ gives the following.
\begin{corollary}
\label{FCAD}
The complete anticanonical linear system  of
$\Bbb{F}_{n}$ does not have a fixed component if $n$ takes the
values  zero, one or two. And, it has $C_n$ as the fixed component
for $n\geq 3$.
\end{corollary}
\begin{proof}
Taking into account that the complete linear system of the
anticanonical class of $\Bbb{F}_{0}$, $\Bbb{F}_{1}$ and
$\Bbb{F}_{2}$  respectively are $|2C_0+2F|$, $|2C_n+3F|$ and
$|2C_n+4F|$ respectively; and these complete linear systems
contains integral curves, the result holds in the case of
$\Bbb{F}_{n}$ with $0\leq n\leq 2$. Now, assume that $n\geq 3$.
From $C_n.(2C_n+(2+n)F)=2-n<0$, we deduce that $C_n$ is a fixed
component of the complete linear system $|2C_n+(2+n)F|$ of the
anticanonical class of $\Bbb{F}_{n}$. On the other hand, since the
complete linear system $|C_n+(2+n)F|$ contains a smooth curve, we
deduce that $C_n$ is the fixed component of $|2C_n+(2+n)F|$.
\end{proof}

\section{The Case of a blow up a Geometrically Ruled Surface}
Here, we consider the case when the Picard number $\rho(Y)$ of the
anticanonical rational surface  $Y$ is greater than or equal to
three. In the following theorem, we determine in particular the
fixed components of the anticanonical complete linear system of
$Y$ if it has some. If this system does not have any, then the
nature of $Y$ can be also determined.

\begin{theorem}
\label{FCAD2} Let $Y$ be an anticanonical rational surface with
Picard number $\rho(Y)\geq 3$. Two cases may occur:
\begin{itemize}
\item [(1-)] If the anticanonical complete linear system $|-K_{Y}|$ has a fixed
component, then it is either a $(-n)$-curve or an integral curve
of arithmetic genus equal to one and of self-intersection less
than or equal to zero. Moreover, the second case occurs with an
integral curve of self-intersection equal to zero only if
$K_{Y}^{2}=0$.
\item [(2-)] If the anticanonical complete linear system $|-K_{Y}|$ does
not have a fixed component, then  $K_{Y}^{2}\geq 0$ and   $Y$ is
isomorphic to  a blow up the projective plane at $r$ points, may
be infinitely near, $r$ is an integer less than or equal to nine.

\end{itemize}
\end{theorem}

\begin{proof}
Since a blow up of $\Bbb{F}_{0}$ or of $\Bbb{F}_{1}$ at a nonempty
set of points (may be infinitely near) has  the projective plane
$\Bbb{P}^{2}$ as a minimal model, and  since a blow up of
$\Bbb{F}_{2}$ at a nonempty set of points (may be infinitely near)
has $\Bbb{P}^{2}$ or $\Bbb{F}_{3}$ as a minimal model, we may
assume that the surface $Y$ has  either $\Bbb{P}^{2}$ or
$\Bbb{F}_{n}$, with $n\geq 3$, as a minimal model.\\
Let us prove the item $(1-)$. Assume first that $\Bbb{P}^{2}$ is a
minimal model of $Y$ and let $\phi $ be a projective birational
morphism from $Y$ to $\Bbb{P}^{2}$. Let $\Gamma $ be a fixed
irreducible component of the complete linear system $|-K_{Y}|$.
Two possibilities may occur: $\phi (\Gamma )$ is either a point of
$\Bbb{P}^{2}$ or an integral  curve on $\Bbb{P}^{2}$.\\
Assume that $\phi (\Gamma )$ is a point, then by \cite[Exercise 5.4.
(a), page 419]{HA}, we deduce that $\Gamma $ is a smooth rational
curve of self-intersection strictly negative, i.e. a $(-n)$-curve on
$Y$ for some integer $n\geq 1$. Now  assume that $\phi (\Gamma )$ is
an irreducible curve on $\Bbb{P}^{2}$, let denote by $d$ its degree.
Since $\phi ( -K_{Y}) $ has degree equal to three. It follows that
$1\leq d \leq 3$. If $d=3$, then we have $\Gamma +\sum_{i=1}^{i=u}
n_{i}E_{i}= -K_{Y}$ for some integers $n_{i}\geq 0$ and some smooth
rational curves $E_{i}$ of self-intersection strictly negative,
where $u\geq 1$ is an integer. On the other hand, it follows from
the fact that $\Gamma $ is a fixed irreducible component of
$|-K_{Y}|$ that ${\Gamma ^{2}\leq 0} $. Otherwise, we would get that
$\Gamma ^{2} >0$, in particular $\Gamma $ (see Lemma \ref{move})
moves which is a contradiction
with the fact that $\Gamma $ does not move.\\
If $d=2$, then $\phi (\Gamma )$ is an irreducible conic on
$\Bbb{P}^{2}$. Hence, it is a smooth rational curve. It follows
from \cite[Corollary 5.4., page 411]{HA} that $\Gamma $ is also a
smooth rational curve on $Y$. And  $\Gamma $ should be of
self-intersection strictly negative. The same argument prove that
if $d=1$, then $\Gamma $ is a smooth rational curve of
self-intersection strictly negative.\\
Now let $n\geq 3$ be a fixed integer, assume that $\Bbb{F}_{n}$ is
a minimal model of $Y$. Then consider $\psi $ be a projective
birational morphism from $Y$ to  $\Bbb{F}_{n}$. Let $\Gamma $ be a
fixed irreducible component of $|-K_{Y}|$, then we can assume that
$\psi (\Gamma)$ is an irreducible curve on $\Bbb{F}_{n}$.
Otherwise, it should be a point of $\Bbb{F}_{n}$; so by proceeding
as in the above case for  $\Bbb{P}^{2}$, we get the result.\\
Thus  assuming that $\psi (\Gamma)$ is an irreducible curve, in
particular, it  is an irreducible component of $-K_{\Bbb{F}_{n}}=
2C_{n}+ (2+n)F$. Thus taking into account of the results of
Proposition \ref{FCGRSED} , $\Gamma $ may be one of the following
irreducible curves: $C_{n}$, $F$, $C_{n}+ nF$,
$C_{n}+ (1+n)F$, and $C_{n}+(2+n)F$. So the result follows.\\

The item $(2-)$ follows at once by remarking that an anticanonical
divisor  $-K_{Y}$ of $Y$ is numerically effective.

\end{proof}

In particular, the following result holds:

\begin{corollary}
\label{FCAD3}
Let $X$ be a smooth rational surface such that
$K_{X}^{2}\geq 0$, where $K_{X}$ denotes a canonical divisor on
$X$. Assume that the anticanonical complete linear system has a
fixed component $\Gamma $. Two cases may occur:
\begin{itemize}
\item If $K_{X}^{2}> 0$, then $\Gamma $ is a $(-n)$-curve, where
$n\geq 1$ is an integer;
\item If $K_{X}^{2}= 0$, then  $\Gamma $ is either an integral curve
of arithmetic genus equal to one  and of self-intersection equal
to zero, or a smooth rational curve of strictly negative
self-intersection.
\end{itemize}
\end{corollary}

Another useful result, see for instance \cite{L0} and \cite{L3},
is:

\begin{corollary}
\label{FCAD4} Let $X$ be a smooth rational surface such that
$K_{X}^{2}\geq 0$, where $K_{X}$ denotes a canonical divisor on
$X$. If $-K_{X}$ is not numerically effective, then the
anticanonical complete linear system has a $(-n)$-curve, $n$ being
an integer greater than or equal to three, as a fixed component.
\end{corollary}

\section*{Acknowledgments} Research partially supported in 2011 by
C.I.C. - U.M.S.N.H. (Mexico).


\begin{thebibliography}{99}
%
\bibitem{BPV} {\sc W. Barth}, {\sc C. Peters} and  {\sc A. Van de Ven.}, {\em
Compact Complex Surfaces}, Berlin, Springer,  1984.
%
\bibitem{B} {\sc A. Beauville}, {\em  Surfaces Alg\'ebriques Complexes},
Ast\'erisque 54. Soci\'et\'e Math\'ematique de France,  1978.
%
\bibitem{FLM0} {\sc G. Failla}, {\sc M. Lahyane} and {\sc G. Molica Bisci}, {\em On the finite generation of the monoid of effective divisor
classes on rational surfaces of type (m, n)}, Atti della Accademia Peloritana dei Pericolanti Classe di Scienze Fisiche, Matematiche e Naturali, LXXXIV, 1–9 (2006).
%
\bibitem{FLM1} {\sc G. Failla}, {\sc M. Lahyane} and {\sc G. Molica Bisci}, {\em 
Rational surfaces of Kodaira type IV}, Bollettino della Unione Matematica Italiana. Serie VIII. Sezione B. Articoli di Ricerca Matematica,  Volume {\bf 10} (2007),  number 3, 741 - 750.
%
\bibitem{FLM2} {\sc G. Failla},  {\sc M. Lahyane} and {\sc G. Molica Bisci}, {\em The finite generation of the monoid of effective divisor classes on Platonic rational surfaces.} Singularity theory. pp.  565 - 576, World Sci. Publ., Hackensack, NJ, 2007.
%
\bibitem{GM1} {\sc C. Galindo}, and {\sc F. Monserrat}, {\em On the cone of curves and of line bundles of a rational surface.}  International  Journal of  Mathematics, Volume  {\bf 15}  (2004),  number 4, 393 - 407.
%
\bibitem{GM2} {\sc C. Galindo}, and {\sc F. Monserrat}, {\em The cone of curves associated to a plane configuration}  Commentarii Mathematici Helvetici, Volume  {\bf 80}  (2005),  number 1, 75 - 93.

%
\bibitem{Harb2} {\sc B. Harbourne},  {\em Rational surfaces with $K^{2}>0$,}
Proceedings of the American Mathematical Society, Volume 124,
Number 3 (March 1996).
%
\bibitem{HA} {\sc R. Hartshorne},    {\em  Algebraic Geometry}, Graduate Texts in
Mathematics, Springer Verlag,  1977.
%

\bibitem{L0} {\sc M. Lahyane},    {\em  Exceptional curves on rational surfaces
having $K^{2}\geq 0$}. Comptes Rendus de l'Acad\'emie des
Sciences-Series I-Mathematics, {\bf 338} (2004), 873 - 878.
%
\bibitem{L1} {\sc M. Lahyane},  {\em Irreducibility of the $(-1)$-Classes on
Smooth Rational Surfaces}. Proceedings of the American Mathematical
Society 133 (2005) 1593-1599.
%
\bibitem{L2} {\sc M. Lahyane},  {\em Rational Surfaces Having Only
a Finite Number of Exceptional Curves}. Mathematische Zeitschrift,
Volume {\bf 247}, Number 1, 213 - 221 (May 2004).
%
\bibitem{L3}
{\sc M. Lahyane}, {\em Exceptional Curves on Smooth Rational
Surfaces with $-K$ not Nef and of Self-intersection Zero}.
Proceedings of the American Mathematical Society, {\bf 133} (2005),
1593 - 1599.
%
\bibitem{LH} {\sc M. Lahyane} and {\sc B. Harbourne}, {\em Irreducibility of $-1$-classes on anticanonical rational surfaces and finite generation of the
effective monoid}, Pacific Journal of Mathematics,  Volume {\bf 218}
Number 1 (2005), pp. 101 - 114.
%
\bibitem{L4} {\sc M. Lahyane}, 
{\em On the finite generation of the effective monoid of rational surfaces}, Journal of Pure and Applied Algebra,  Volume {\bf 214}  (2010),  number 7, 1217–1240.
%
\bibitem{JR} {\sc J. Rosoff},  {\em Effective divisor classes and blowings-up of
${\Bbb P}^{2}$}. Pacific Journal of Mathematics,  Volume {\bf 89},
No.2, 1980.
%
\bibitem{UPRM} {\sc R. Miranda, U. Persson},  {\em On Extremal Rational
Elliptic Surfaces}. Mathematische Zeitschrift {\bf 193}, 537-558
(1986).
%
\bibitem{MN} {\sc M. Nagata}, {\em On rational surfaces, II}, Mem. Coll. Sci.
Univ. Kyoto,  Ser. A Math.  {\bf 33}  (1960), 271-293.
%


\end{thebibliography}
\end{document}